\documentclass[12pt]{amsart}
\usepackage{amsmath}
\usepackage{amssymb,amscd}
\usepackage[colorlinks=true, urlcolor=blue,bookmarks=true,bookmarksopen=true, citecolor=blue,hypertex]{hyperref}

\numberwithin{equation}{section}

\newtheorem{defn}{Definition}[section]
\newtheorem{theorem}{Theorem}[section]

\newtheorem{corollary}[theorem]{Corollary}

\newtheorem{lemma}[theorem]{Lemma}

\newtheorem{remark}[theorem]{Remark}

\def \begineq{\begin{equation}}
\def \endeq{\end{equation}}

\def \bb{\mathbb}

\def \CC{{\bb{C}}}

\def \PP{{\bb{P}}}
\def \QQ{{\bb{Q}}}
\def \RR{{\bb{R}}}

\def \ZZ{{\bb{Z}}}

\def \({\left(}
\def \){\right)}
\def \<{\langle}
\def \>{\rangle}
\def \bar{\overline}
\def \deg{\mathrm{deg}}

\begin{document}

\title[ Quasitoric orbifolds ]{Almost complex structure, blowdowns and McKay
 correspondence in  quasitoric orbifolds}

\author[Saibal Ganguli]{Saibal Ganguli}

\address{Departamento de
Matem\'aticas, Universidad de los Andes, Bogota, Colombia}

\email{saibalgan@gmail.com}

\author{Mainak Poddar}
\address{ Stat-Math Unit, Indian
Statistical Institute, Kolkata, India; and Departamento de
Matem\'aticas, Universidad de los Andes, Bogota, Colombia}
%\thanks{}

\email{mainakp@gmail.com}


%MSC classification

\abstract
 We prove the existence of  invariant almost complex
 structure on any positively omnioriented
 quasitoric orbifold. We construct blowdowns. We define
 Chen-Ruan cohomology ring for any omnioriented quasitoric orbifold.
 We prove that the Euler characteristic of this cohomology
 is preserved by a crepant blowdown. We prove that the Betti numbers are
 also preserved if dimension is less or equal to six. In particular, our work
 reveals a new form of McKay correspondence for orbifold toric varieties
 that are not Gorenstein. We illustrate with an example.
 \endabstract

 \maketitle

 \section{Introduction}\label{intro}

 McKay correspondence \cite{[Re]} has been studied widely for complex
 algebraic varieties with only Gorenstein or $SL$ orbifold singularities. A cohomological
  version of this correspondence says that the Hodge numbers (and Betti
numbers) of Chen-Ruan
 cohomology (with compact support) \cite{[CR]} are preserved under crepant blowup.
  This was proved
 in \cite{[LP]} and \cite{[Yas]} for complete algebraic varieties with $SL$ quotient
 singularities following fundamental work of \cite{[Bat]} and \cite{[DL]} in the local
 case. It makes sense to ask if such a correspondence holds for Betti
  numbers when the orbifold has almost complex structure only.
   However the main ingredients in the algebraic proof,
 namely motivic integration and Hodge structure, may no longer
 be available.

 From a different perspective, the topological properties of
 quasitoric spaces introduced by
 Davis and Januskiewicz \cite{[DJ]}, have been studied
 extensively. However not much attention has been given to the
 study of equivariant maps between them. In this article, which is a
 sequel to  \cite{[GP]}, we construct equivariant blowdown maps between
 primitive omnioriented quasitoric orbifolds and prove certain McKay
 type correspondence for them.
  These spaces do not have complex or almost complex structure in general.

Quasitoric orbifolds \cite{[PS]}  are topological generalizations
of projective simplicial toric varieties or symplectic toric
orbifolds \cite{[LT]}. They are even dimensional spaces with
action of the compact torus of half dimension such that the orbit
space has the structure of a simple polytope.
 We only work with primitive quasitoric
orbifolds. The orbifold singularities of these spaces correspond
to analytic singularities.
   An omniorientation is a choice of orientation
  for the quasitoric orbifold as well
 as for each invariant suborbifold of codimension two.
 When these orientations are compatible the quasitoric orbifold is called
 positively omnioriented, see section \ref{omnio} for details.  We
prove the existence of invariant almost complex structure on
positively omnioriented quasitoric orbifolds (Theorem
\ref{thmacs}) by adapting the technique of Kustarev \cite{[Kus]}
for quasitoric manifolds.
 We also build a stronger version of Kustarev's result: Theorem \ref{extn} and
Corollary \ref{extc}. These may be of use to even those who are
mainly interested in quasitoric manifolds.

 Chen-Ruan cohomology was originally defined for almost complex orbifolds in
\cite{[CR]}. There the almost complex structure on normal bundles
of singular strata is used to determine the grading of the
cohomology. An omniorientation, together with the torus action,
determines a complex structure on the normal bundle of every
invariant suborbifold of a quasitoric orbifold. Moreover the
singular locus is a subset of the union of invariant suborbifolds.
Thus we can define Chen-Ruan cohomology groups for any
omnioriented quasitoric orbifold, see section \ref{crc}. We also
define a ring structure for this cohomology in section \ref{ring}
following the approach of \cite{[CH]}. The Chen-Ruan cohomology of
the same quasitoric orbifold is in general different for different
omniorientations. For a positively omnioriented quasitoric
orbifold with the almost complex structure of Theorem
\ref{thmacs}, our definition of Chen-Ruan cohomology ring agrees
with that of \cite{[CR]}.

The blowdown maps are continuous, and they are diffeomorphism of
orbifolds away from the exceptional set.  They are not morphisms
of orbifolds (see \cite{[ALR]} for definition).  In some cases
they are analytic near the exceptional set, see Lemma
\ref{pseubd}. (In these cases they are pseudoholomorphic in a
natural sense, see Definition \ref{smfn}.) For these  we can
compute the pull-back of the canonical sheaf and test if  the
blowdown is crepant in the sense of complex geometry: The pull
back of the canonical sheaf of the blowdown is the canonical sheaf
of the blowup. However the combinatorial condition this
corresponds to, makes sense in general and may be applied to an
arbitrary blowdown. We work with this generalized notion of
crepant blowdown, see section \ref{crepbd}.

We prove the conservation of Betti numbers of Chen-Ruan cohomology
under crepant blowdowns when the quasitoric orbifold has dimension
less than or equal to six (Theorem \ref{thmmckay6}).
 We also prove the conservation of Euler characteristic of this
 cohomology under crepant blowdowns in arbitrary dimension
 (Theorem \ref{thmeuler}). This implies that the rational orbifold
  $K$-groups \cite{[AR]} are also preserved, see section \ref{kgps}.
  These statements hold under the condition that
 the  omnioriented quasitoric orbifolds are quasi-$SL$, a
  generalization of $SL$; see Definition \ref{quasisl}.

  The validity of McKay correspondence for
Betti numbers remains an interesting open problem in higher
dimensions. One might try to make use of the local results from
motivic integration, namely correspondence of Betti numbers of
Chen-Ruan cohomology with compact support for crepant blowup of a
Gorentstein quotient singularity $\CC^n/G$ \cite{[Bat],[DL]}.
However such efforts are impeded by the fact that the
correspondence obtained from motivic integration is not natural.
 However, we prove a very basic inequality about the
behavior of the second Betti number under crepant blowup in Lemma
\ref{thmh2}. We also give an example of McKay correspondence for
Betti numbers when dimension is eight in section \ref{exapl}. This
example is particularly interesting as it corresponds to the
weighted projective space $\PP(1,1,3,3,3)$ which is not a
Gorenstein or $SL$ orbifold. Hence McKay correspondence as studied
in complex algebraic geometry does not apply to it. However under
suitable choice of omniorientation it is quasi-$SL$ and McKay
correspondence holds. Note that the blowup is not a toric blowup
in the sense of algebraic geometry.

 In \cite{[GP]}, we constructed examples
of four dimensional quasitoric orbifolds that are not toric
varieties. We also constructed pseudoholomorphic blowdowns between
them. Our brief study of pseudo-holomorphicity of blowdowns
 in section \ref{pseudo} shows that every primitive positively
omnioriented quasitoric orbifold of dimension four has a
pseudoholomorphic resolution of singularities, see Theorem
\ref{thmresol}. The result may hold in dimension six as well, but
developing pseudoholomorphic blowdowns in dimension six and higher
would need further work.

 \section{Quasitoric orbifolds}\label{smooth}

 In this section we
review the combinatorial construction  of quasitoric
orbifolds. We also  construct an explicit
orbifold atlas for them and list a few important properties.
 The notations established here will be important for
the rest of the article.

\subsection{Construction}   Fix a copy $N$ of $\ZZ^n$ and let $ T_N := (N \otimes_{\ZZ} \RR) / N \cong \RR^n/ N $
 be the corresponding $n$-dimensional torus. A primitive vector in $N$, modulo sign,
  corresponds to a circle subgroup.
 of $T_N$. More generally, suppose $M$ is a submodule of $N$ of rank $m$. Then
 \begin{equation}\label{TM} T_M := (M \otimes_{\ZZ} \RR) /M \end{equation}
  is a torus of dimension $m$.
  Moreover there is a natural homomorphism of Lie
 groups $\xi_M: T_M \to T_N$ induced by the inclusion $M \hookrightarrow N$.
\begin{defn}\label{tlambda}  Define T(M) to be the image
  of $T_M$ under $\xi_M$.
 If $M$ is generated  by a vector $\lambda \in N$, denote $T_M$
 and $T(M)$ by $T_{\lambda}$ and $T(\lambda)$ respectively.
\end{defn}

Usually a polytope is defined to be the convex hull of a finite
set of points in $\RR^n$. To keep our notation manageable, we will
take a more liberal interpretation of the term polytope.
\begin{defn}\label{polytope} A polytope $P$ will denote a subset of $\RR^n$ which
is diffeomorphic, as manifold with corners, to the convex hull $Q$
of a finite number of points in $\RR^n$. Faces of $P$ are the
images of the faces of $Q$ under the diffeomorphism.
\end{defn}

Let $P$ be a simple polytope in $\RR^n$, i.e. every vertex of $P$
is the intersection of exactly $n$ codimension one faces (facets).
Consequently every $k$-dimensional face $F$ of $P$ is the intersection of
 a unique collection of $n-k$ facets.
 Let $ \mathcal{F}:= \{F_1,\ldots, F_m\}$ be the
set of facets of $P$.

\begin{defn}\label{charfn}
 A function $\Lambda: \mathcal{F} \to N $
  is called a characteristic function for $P$ if
 $ \Lambda(F_{i_1}), \ldots, \Lambda(F_{i_k})$ are linearly independent whenever
 $F_{i_1}, \ldots, F_{i_k}$intersect at a face in $P$.  We write $\lambda_i$
 for $\Lambda(F_i)$ and call it a characteristic vector.
\end{defn}

\begin{remark}\label{primi} In this article we  assume that all characteristic
 vectors are primitive.
Corresponding quasitoric orbifolds have been termed primitive
quasitoric orbifold in \cite{[PS]}. They are characterized by the
codimension of singular locus being greater than or equal to four.
\end{remark}

 \begin{defn}\label{NF}
 For any face $F$ of $P$, let $\mathcal{I}(F) = \{ i  | F \subset F_i  \}$.
Let $\Lambda$ be a characteristic function for P.  Let $N(F)$ be the submodule of $N$ generated
by $\{ \lambda_i:  i \in \mathcal{I}(F) \} $. Note that $\mathcal{I}(P) $ is empty  and
$N(P) = \{0\}$. \end{defn}

  For any point $p \in P$, denote by $F(p)$
 the face of  $P$ whose relative interior contains $p$. Define an equivalence relation $\sim$ on the space
 $P \times T_N$ by
 \begin{equation} \label{defequi}
 (p,t) \sim (q,s) \; {\rm if \; and\; only\; if\;} p=q \; {\rm and} \; s^{-1}t \in T(N({F(p)}))
 \end{equation}
 Then the quotient space $X :=P \times T_N/ \sim$ can be given the structure of a $2n$-dimensional
  quasitoric orbifold. Moreover any  $2n$-dimensional primitive quasitoric orbifold may be obtained in
  this way, see \cite{[PS]}. We refer to the pair $(P,\Lambda)$ as a model for the
  quasitoric orbifold. The space $X$ inherits an action of $T_N$ with orbit space
   $P$ from the natural action
  on $P \times T_N$. Let $\pi: X \to P$ be the associated quotient map.

   The space $X$ is a manifold if the characteristic vectors $ \lambda_{i_1}, \ldots, \lambda_{i_k}$
   generate a unimodular subspace of $N$ whenever the facets $F_{i_1}, \ldots, F_{i_k}$ intersect.
    The points $\pi^{-1}(v) \in X$, where $v$ is any vertex of $P$,
are fixed by the action of $T_N$. For simplicity we will denote
the point $\pi^{-1}(v)$  by $v$ when there is no confusion.

\subsection{Orbifold charts}\label{diffs}
Consider open
 neighborhoods $U_v \subset P$ of the vertices $v$  such that
  $U_v$ is the complement in $P$ of all edges
 that do not contain $v$.
 Let
 \begin{equation}
X_v := \pi^{-1}(U_v) =  U_v \times T_N / \sim
\end{equation}
For a face $F$ of $P$ containing $v$ there is a natural inclusion of
$N(F)$ in $N(v)$.
  It induces an injective homomorphism $T_{N(F)} \to T_{N(v)}$ since a basis of $N(F)$ extends
  to a basis of $N(v)$. We will regard $T_{N(F)}$ as a
  subgroup of $T_{N(v)}$ without confusion.
Define an equivalence relation $\sim_v$ on $U_v \times T_{N(v)}$ by
$(p,t)\sim_v (q,s)$ if
 $p=q$ and $s^{-1}t \in T_{N(F)}$ where $F$ is the face whose relative interior contains $p$.
Then the space
\begin{equation}
 \widetilde{X}_v:= U_v \times T_{N(v)}/ \sim_v
\end{equation}
is $\theta$-equivariantly diffeomorphic to an open set in $\CC^n$,
where  $\theta: T_{N(v)} \to U(1)^n$ is an isomorphism, see \cite{[DJ]}. This means that
   there exists a diffeomorphism $f: \widetilde{X}_v \to B \subset \CC^n  $ such that
   $f(t\cdot x) = \theta(t) \cdot f(x)$ for all $x \in \widetilde{X}_v $. This will be
   evident from the subsequent discussion.

The map $\xi_{N(v)} : T_{N(v)} \to T_N$ induces a map $\xi_v:
\widetilde{X}_v \to X_v$ defined by
   $\xi_v([(p,t)]^{\sim_v}) = [(p,\xi_{N(v)}(t)) ]^{\sim}$ on equivalence classes.
    The kernel of $\xi_{N(v)}$,  $G_v = N/N(v)$, is a finite
  subgroup of $T_{N(v)}$ and therefore has a natural smooth, free action on $T_{N(v)}$
  induced by the group operation.  This induces smooth action of $G_v$ on
  $\widetilde{X}_v$. This action is not free in general. Since $T_N \cong T_{N(v)}/G_v $,  $X_v$
 is homeomorphic to the quotient space $\widetilde{X}_v/G_v$.  An orbifold chart
 (or uniformizing system) on $X_v$ is
  given by $(\widetilde{X}_v, G_v, \xi_v)$.

 Let $(p_1, \ldots, p_n)$ denote the standard
    coordinates on  $\RR^n \supset P$.
  Let $(q_1, \ldots, q_n)$ be the coordinates on $N \otimes \RR$ that correspond to the standard
  basis of $N$.
 Let $\{u_1, \ldots, u_n\}$ be the standard basis of $N$.
 Suppose the characteristic vectors $u_i$ are assigned to the facets $p_i=0$ of the
 cone $ \RR^n_{\ge} $.
     In this case  there is a homeomorphism $\phi: (\RR^n_{\ge}
     \times T_N/\sim ) \to  \RR^{2n}$
    given by
     \begin{equation}
    x_i= \sqrt{p_i} \cos (2 \pi q_i), \;\;
    y_i = \sqrt{p_i} \sin (2 \pi q_i) \;\; {\rm where} \; i=1, \ldots, n.
    \end{equation}

\begin{remark} The square root over $p_i$ is necessary to ensure that the
orbit map $\pi: \RR^{2n} \to \RR^n_{\ge} $ is smooth.
\end{remark}

      We  define a homeomorphism
      $\phi_{v}: \widetilde{X}_v  \to \RR^{2n}$ as follows.
      Assume without loss of generality that $F_1, \ldots, F_n$ are the facets of
      $U_v$. Let the equation of $F_i$ be $p_{i,v}= 0$.
     Assume that $p_{i,v} > 0$ in the interior of $U_v$ for every $i$.
 Let $\Lambda_{v}$ be the corresponding matrix of characteristic
 vectors \begin{equation}\label{lambdav}
 \Lambda_{v}= [\lambda_{1} \ldots \lambda_{n} ]. \end{equation}

     If ${\bf q}_{v}= (q_{1,v},\ldots, q_{n,v})^t$ are angular coordinates of an element of $T_N$ with respect
     to the basis $\{ \lambda_1, \ldots, \lambda_n \}$ of $N \otimes \RR$, then  the
     standard coordinates ${\bf q} =(q_1, \ldots, q_n)^t$ may be expressed as
     \begin{equation}\label{chcoord1}
     {\bf q} = \Lambda_{v} {\bf q}_{v}.
     \end{equation}
     Then define the homeomorphism $\phi_{v}: \widetilde{X}_v \to \RR^{2n}$ by
     \begin{equation}\label{homeo}
      x_i = x_{i,v}:= \sqrt{p_{i, v}} \cos(2 \pi q_{i,v} ), \quad
      y_i = y_{i,v}:= \sqrt{p_{i,v} } \sin( 2 \pi q_{i,v} ) \quad {\rm for}\;
      i=1,\ldots,n
     \end{equation}
We write
\begin{equation}\label{cplxcoor}
 z_i = x_i + \sqrt{-1} y_i, \quad {\rm and} \quad
z_{i,v} = x_{i,v} + \sqrt{-1}y_{i,v}
\end{equation}

 Now consider the action of $G_v = N/N(v)$ on $\widetilde{X}_v$. An element
     $g$ of $G_v$ is represented by a vector $\sum_{i=1}^n a_i \lambda_i $ in
    $N$ where each $a_i \in  \QQ$.  The action of $g$ transforms the coordinates
    $q_{i,v}$ to $q_{i,v} + a_i$. Therefore
    \begin{equation}\label{action}
     g\cdot (z_{1,v},\ldots, z_{n,v}) = (e^{2\pi \sqrt{-1} a_1} z_{1,v},\ldots, e^{2\pi \sqrt{-1}a_n} z_{n,v}).
     \end{equation}

We may identify $G_v$ with the cokernel of the linear map
$\Lambda_{v}: N \to N$. Then standard arguments using the Smith
normal form of the matrix $\Lambda_{v}$ imply that
\begin{equation}\label{ordgv} o(G_v)= |\det \Lambda_{v} |.
\end{equation}

\subsection{Compatibility of charts}\label{comp}
We show the compatibility of the charts $(\widetilde{X}_v, G_v,
\xi_v) $. Let $v_1$ and $v_2$ be two vertices so that the minimal
face $S$ of $P$ containing both  has dimension $s \ge 1$. Then
$X_{v_1} \cap X_{v_2}$ is nonempty.
 Assume facets $(F_1,\ldots, F_s, F_{s+1}, \ldots, F_n)$ meet at vertex $v_1$ and
 facets $( F_{n+1} , \ldots, F_{n+s}, F_{s+1},\ldots, F_n)$ meet at $v_2$.
We take \begin{equation} \begin{array}{l}
 \Lambda_{v_1}= [\lambda_1, \ldots, \lambda_s,
\lambda_{s+1}, \ldots, \lambda_n] \; {\rm and} \\
 \Lambda_{v_2}= [\lambda_{n+1}, \ldots, \lambda_{n+s},\lambda_{s+1}, \ldots, \lambda_n] .
\end{array} \end{equation}
 Then
\begin{equation}\label{chcoord2}
  {\bf q}_{v_2} = \Lambda_{v_2}^{-1} \Lambda_{v_1} {\bf
  q}_{v_1}
\end{equation}
 Suppose \begin{equation}\label{cjk}
 \lambda_k = \sum_{j=s+1}^{n+s} c_{j,k} \lambda_j, \, 1\le k \le s.
\end{equation}
Then by \eqref{chcoord2},
\begin{equation}
 \begin{array}{ll}
 q_{j,v_2} = \sum_{k=1}^s c_{n+j,k}\, q_{k,v_1}  & {\rm if}\, 1\le j \le s  \\
 q_{j,v_2} = \sum_{k=1}^s c_{j,k}\, q_{k,v_1} + q_{j,v_1} & {\rm if} \, s+1\le j \le n . \\
\end{array}
\end{equation}

 Let the facets $F_j, \, j= 1, \ldots, n+s,$ be defined
by $\widehat{p}_j =0$ such that $\widehat{p}_j > 0$ in the
interior of the polytope $P$. Then the coordinates \eqref{homeo}
on $\widetilde{X}_{v_2}$ and $\widetilde{X}_{v_1}$ are related as
follows.
\begin{equation}\label{chcoord3}
\begin{array}{ll} z_{j,v_2} = \prod_{k=1}^s z_{k,v_1}^{c_{n+j,k}} \sqrt{
\widehat{p}_{n+j} \prod_{k=1}^s \widehat{p_k}^{-c_{n+j,k}} } & {\rm if}\, 1\le j \le s   \\
z_{j,v_2} =  z_{j,v_1} \prod_{k=1}^s z_{k,v_1}^{c_{j,k}}
\sqrt{ \prod_{k=1}^s \widehat{p_k}^{-c_{j,k}} } &
{\rm if} \, s+1\le j \le n. \\
\end{array}
\end{equation}

Take any point $x \in X_{v_1} \bigcap X_{v_2} $. Let
$\widetilde{x}$ be a preimage of $x$ with respect to $\xi_{v_1}$.
Suppose $\pi(x)$ belongs to the relative interior of the face $F
\subset S$. Suppose $F$ is the intersection of facets $F_{i_1},
\ldots, F_{i_t}$ where $ s+1 \le i_1 < \ldots < i_t \le n$. Then
the coordinate $z_{j,v_1}(\tilde{x}) $ is zero if and only if $j
\in \mathcal{I}(F)=\{i_1, \ldots, i_t \}$. Consider the isotropy
subgroup $G_x$ of $\tilde{x}$ in $G_{v_1}$. It consists of all
elements that do not affect the nonzero coordinates of
$\tilde{x}$,
\begin{equation}\label{gx1} G_x = \{ g \in G_{v_1} \,:\, g\cdot z_{j,v_1} =
z_{j,v_1}\, {\rm if}\, j \notin \mathcal{I}(F)  \}
\end{equation}
It is clear that $G_x$ is independent of the choice of $\tilde{x}$
and
\begin{equation}\label{gx2}
G_x =\{ [\eta] \in N/N(v_1)\, :\, \eta = \sum_{j\in \mathcal{I}(F)} a_j
\lambda_j \}.
\end{equation}
Note that $j\in \mathcal{I}(F)$ if and only if $\lambda_j \in N(F)$. It
follows from the linear independence of $\lambda_1, \ldots,
\lambda_n$ that
\begin{equation}\label{gx3}
G_x \cong G_F := ((N(F)\otimes_{\ZZ} \QQ) \cap N) / N(F).
\end{equation}
Note that $G_P$ is the trivial group.

 Choose a small ball $B(\widetilde{x},r)$ around
$\widetilde{x}$ such that $(g\cdot B(\widetilde{x},r)) \bigcap
B(\widetilde{x},r)$ is empty for all $g \in G_{v_1}-G_x$. Then
 $B(\widetilde{x},r)$ is stable under the action of $G_x$ and
  $(B(\widetilde{x},r), G_{x},\xi_{v_1})$ is an orbifold chart around
$x$ induced by  $(\widetilde{X}_{v_1}, G_{v_1}, \xi_{v_1})$. We
show that for sufficiently small value of $r$, this chart embeds
into $(\widetilde{X}_{v_2}, G_{v_2}, \xi_{v_2})$ as well.

Note that the rational numbers $c_{j,k}$ in \eqref{cjk} are
integer multiples of $\frac{1}{\Delta}$ where $\Delta =
\det(\Lambda_{v_2} )$. Choose a branch of
$z_{k,v_1}^{\frac{1}{\Delta}}$ for each $1\le k\le s$, so that
the branch cut does not intersect $ B(\widetilde{x},r)$. Assume
$r$ to be small enough so that the functions
$z_{k,v_1}^{c_{j,k}}$ are one-to-one on $ B(\widetilde{x},r)$
for each $s+1 \le j \le n+s$ and $1\le k\le s$. Then equation
\eqref{chcoord3} defines a smooth embedding $\psi$ of $
B(\widetilde{x},r)$ into $\widetilde{X}_{v_2}$. Note that
$\widehat{p_k}, \, 1\le k \le s,$ and $\widehat{p}_{n+j}, \, 1\le
j \le s$ are smooth non-vanishing functions on
 $\xi_{v_1}^{-1} ( X_{v_1} \bigcap X_{v_2})$.  Let
  $i_{v_2}: G_x \to G_{v_2} $ be the natural inclusion
obtained using equation \eqref{gx3}. Then $(\psi, i_{v_2}):(
B(\widetilde{x},r), G_{x},\xi_{v_1}) \to
 (\widetilde{X}_{v_2}, G_{v_2}, \xi_{v_2}) $ is an
 embedding of orbifold charts.

We denote the space $X$ with the above orbifold structure by ${\bf
X}$. In general we will use a boldface letter to denote an
orbifold and the same letter in normal font to denote the
underlying topological space.

\subsection{Independence of shape of polytope}

 \begin{lemma}\label{class} Suppose ${\bf X}$ and ${\bf Y}$ are  quasitoric
    orbifolds whose orbit spaces
    $P$ and $Q$ are diffeomorphic and the characteristic vector of any edge of $P$
    matches with the characteristic vector of the corresponding edge of $Q$. Then ${\bf X}$
    and ${\bf Y}$ are equivariantly diffeomorphic.
    \end{lemma}

 \begin{proof} Pick any vertex $v$ of $P$.  For simplicity we will write $p_i$ for $p_{i,v}$, and
$q_i$ for $q_{i,v}$. Suppose the diffeomorphism $f : P_1 \to
P_2$ is given near $v$ by $f(p_1,p_2 \ldots ,p_n)= (f_1,f_2 \ldots
,f_n )$. It induces a map of
 local charts  $\widetilde{X}_v \to \widetilde{Y}_{f(v)} $ by
\begin{equation}\label{map}
(\sqrt{p_i} \cos(2\pi q_i), \sqrt{ p_i} \sin(2\pi q_i) ) \mapsto
(\sqrt{f_i} \cos(2\pi q_i), \sqrt{f_i} \sin(2\pi q_i))  \quad {\rm
for}\; i= 1,\ldots,n.
\end{equation}
This is a smooth map if the functions $\sqrt{f_i/p_i}$ are smooth
functions of $p_1, \ldots, p_n$. Without loss of generality let us
consider the case of $\sqrt{f_1/p_1}$. We may write
\begin{equation}\label{taylor}
f_1(p_1,p_2 \ldots p_n) = f_1(0,p_2 \ldots p_n) +  p_1
\frac{\partial f_1}{\partial p_1} (0,p_2 \ldots p_n) + p_1^2
g(p_1,p_2 \ldots p_n)
\end{equation}
where $g$ is smooth, see section 8.14 of \cite{[Dieu]}. Note that
$f_1(0,p_2\ldots p_n)=0$ as $f$ maps the facet $p_1=0$ to the
facet $f_1 =0$. Then it follows from equation \eqref{taylor} that
$f_1/p_1$ is smooth. We have
\begin{equation}\label{taylor2}
\frac{f_1}{p_1} =  \frac{\partial f_1}{\partial p_1} (0,p_2,\ldots
p_n) + p_1 g(p_1,p_2,\ldots p_n)
\end{equation}
Note that $\frac{f_1}{p_1}$ is nonvanishing away from $p_1=0$.
Moreover we have
\begin{equation}
\frac{f_1}{p_1}=  \frac{\partial f_1}{\partial p_1} (0,p_2\ldots
p_n) \; \; {\rm when}\; p_1 =0.
\end{equation}
Since $f_1(0,p_2 \ldots p_n) $ is identically zero, $
\frac{\partial f_1}{\partial p_j} (0,p_2, \ldots, p_n) =0$ for
each $ 2\le j\le n$. As the Jacobian of $f$ is nonsingular we must
have
\begin{equation}
\frac{\partial f_1}{\partial p_1} (0,p_2, \ldots, p_n) \neq 0
\end{equation}
Thus  $\frac{f_1}{p_1}$ is nonvanishing even when $p_1=0$.
Consequently $\sqrt{f_1/p_1}$ is smooth. Therefore the map
\eqref{map} is smooth and induces an  isomorphism of orbifold
charts.
\end{proof}

\subsection{Torus action} An action of a group $H$ on an orbifold ${\bf Y}$ is an
    action of $H$ on the underlying space $Y$ with some extra conditions. In particular
    for every sufficiently small $H$-stable neighborhood $U$ in $Y$ with uniformizing
    system $(W,G,\pi)$, the action should lift to an action of $H$ on $W$ that commutes
    with the action of $G$. The $T_N$-action on the underlying topological space of a
    quasitoric orbifold does not lift to an action on the orbifold in general.

\subsection{Metric}
    By a torus invariant metric on ${\bf X}$ we will mean a metric on ${\bf X}$ which is
    $T_{N(F)}$-invariant in some uniformizing neighborhood of $x$ for any point
    $x \in \pi^{-1}(F^{\circ})$.

    Any cover of $X$ by $T_N$-stable open sets induces an
    open cover of $P$. Choose a smooth partition of unity on the polytope $P$ subordinate
    to this induced cover.
    Composing with  the projection map $\pi: X \to P$ we obtain a partition of
    unity on $X$ subordinate to the given cover, which is $T_N$-invariant. Such a
    partition of unity is smooth as the map
    $\pi$ is smooth, being locally given by maps $p_j = x_j^2 + y_j^2$.
     For instance, choose a $T_{N(v)}$-invariant metric on each $\widetilde{X}_v$. Then
      using a partition of unity
     as above we can define an invariant metric on ${\bf X}$.

\subsection{Invariant suborbifolds} The $T_N$-invariant subset $X(F) = \pi^{-1}(F)$, where
    $F$ is a face  of $P$, has a natural structure of a quasitoric orbifold \cite{[PS]}.
     This structure is
    obtained by taking $F$ as the polytope for ${\bf X}(F)$ and projecting the characteristic vectors
     to $N/N^{\ast}(F)$ where $N^{\ast}(F)= (N(F) \otimes_{\ZZ} \QQ) \cap N$. With this structure
    ${\bf X}(F)$ is a suborbifold of ${\bf X}$.
    It is called a characteristic suborbifold if $F$ is a facet.
    Suppose $\lambda$ is the characteristic vector attached to the facet $F$. Then
     $\pi^{-1}(F)$ is fixed by the circle subgroup $T(\lambda)$ of $T_N$.
      We denote the relative interior
      of a face $F$ by $F^{\circ}$ and the corresponding invariant space $\pi^{-1}(F^{\circ})$ by $X(F^{\circ})$.
      Note that $v^{\circ}= v$ if $v$ is a vertex.

\subsection{Orientation}
       Note that for any vertex $v$,
      $dp_{i,v} \wedge dq_{i,v} = dx_{i,v} \wedge dy_{i,v}$. Therefore $\omega_{v}:=
      dp_{1,v}\wedge \ldots \wedge dp_{n,v} \wedge
      dq_{1,v} \wedge \ldots \wedge dq_{n,v}$ equals $ dx_{1,v} \wedge \ldots \wedge dx_{n,v} \wedge
      dy_{1,v}
       \wedge \ldots \wedge dy_{n,v} $.
       The standard coordinates $(p_1, \ldots,p_n)$ are related to $(p_{1,v}, \ldots,p_{n,v)}$
      by a diffeomorphism. The same holds  for ${\bf q}$ and ${\bf q}_{v}$. Therefore
      $\omega:= dp_1 \wedge \ldots\wedge dp_n \wedge dq_1 \wedge \ldots \wedge dq_n$ is a nonzero multiple of each
      $\omega_{v}$. The action of $G_v$ on $\widetilde{X}_v$, see equation \eqref{action},
        preserves $\omega_{v}$ for each
      vertex $v$ as $dx_{i,v} \wedge dy_{i,v}= \frac{\sqrt{-1}}{2} dz_{i,v}\wedge d{\bar z}_{i,v}$.
      The action of $G_v$ affects only the angular coordinates. Since
      $dq_1 \wedge \ldots \wedge dq_n = \det(\Lambda_{v}) dq_{1,v} \wedge \ldots \wedge dq_{n,v} $
      and the right hand side is $G_v$-invariant, we conclude that $\omega$ is $G_v$-invariant.
      Therefore $\omega$ defines a nonvanishing $2n$-form on ${\bf X}$.
      Consequently a choice of orientations for $P \subset \RR^n$
      and $T_N$ induces an orientation for ${\bf X}$.

\subsection{Omniorientation}\label{omnio} An omniorientation is a choice of orientation for the
  orbifold as well as an orientation for each characteristic suborbifold.
  For any vertex $v$, there is a representation of $G_v$ on the tangent space${\mathcal T}_0 \widetilde{X}_v$.
  This representation splits into the direct
  sum of $ n$ representations corresponding to the normal spaces of $z_{i,v}=0$.
  Thus we have a decomposition of the orbifold tangent space
  ${\mathcal T}_v{\bf X}$ as a direct sum of the normal spaces of the
  characteristic suborbifolds that meet at $v$.
   Given an omniorientation, we say that the
  sign of a vertex $v$ is positive if the orientations of
   ${\mathcal T}_v({\bf X})$ determined by the
  orientation of ${\bf X}$ and orientations of characteristic suborbifolds coincide. Otherwise
  we say that sign of $v$ is negative. An omniorientation is then said to be positive
   if each vertex has positive sign.

 It is easy to verify that reversing the sign of any number of characteristic vectors does not
 affect the topology or differentiable structure of the quasitoric orbifold. There is a circle
 action of $T_{\lambda_i}$ on the normal bundle of ${\bf X}(F_i)$ producing a complex structure and
 orientation on it. This action and orientation varies with the sign of $\lambda_i$. Therefore,
 given an orientation on ${\bf X}$, omniorientations correspond bijectively to choices of signs
 for the characteristic vectors. We will assume the standard orientations on $P$ and $T^n$ so that
 omniorientations will be solely determined by signs of characteristic vectors.

 At any vertex $v$, we may order the incident facets in such a way that their
 inward normal vectors form a positively oriented basis of $\RR^n \supset P$.
 Facets at a vertex ordered in this way will be called positively ordered.
 We denote the matrix of characteristic vectors ordered accordingly by $\Lambda_{(v)}$. Then
  the sign of $v$ equals the sign of $\det(\Lambda_{(v)})$.

\section{Almost complex structure}\label{acs}

Let ${\bf X}$ be a positively omnioriented primitive quasitoric
orbifold.

\begin{defn} We say that an almost complex structure on ${\bf X}$
  torus invariant if it is $T_{N(F)}$-invariant in some uniformizing neighborhood
   of each point $x \in X(F^{\circ}) $.
 \end{defn}

\begin{theorem}\label{thmacs}
Let ${\bf X}$ be a positively omnioriented quasitoric orbifold and
$\mu$ an invariant metric on it. Then there exists an orthogonal
invariant almost complex structure on ${\bf X}$ that respects the
omniorientation.
\end{theorem}

\begin{proof}
Consider the subset $R_v \subset \widetilde{X}_v $ consisting of
points whose coordinates \eqref{cplxcoor} are real and
nonnegative,
\begin{equation}\label{rv}
R_v = \{ x\in \widetilde{X}_v: z_{j,v}(x) \in \RR_{\ge} \,
\forall 1\le j \le n \}
\end{equation}
In other words,
\begin{equation}
R_v = \{ x\in \widetilde{X}_v: z_{j,v}(x) = \sqrt{p_{j,v}(x)},
\; j= 1, \ldots, n \}
\end{equation}
We glue the spaces $R_v$ according to the transition maps
 \eqref{chcoord3}, choosing
the branches  uniformly as $-\pi < q_{k,v} < \pi$. We obtain a
manifold with boundary $R$.

 Let $x$ be any point in $R_{v_1}$ such that $\xi_{v_1}(x)\in X_{v_1} \cap
 X_{v_2} $.  Then the transition maps
\eqref{chcoord3}, with above choice of cuts, define a local
diffeomorphism $\phi_{12}$ from a neighborhood of $x$ in
$\widetilde{X}_{v_1}$ to a neighborhood of the image of $x$ in
$\widetilde{X}_{v_2}$.

 Let $\mathcal{E}_v$ denote the restriction of
$\mathcal{T}\widetilde{X}_v$ to $R_v$. The last paragraph shows
that these bundles glue to form a smooth rank $2n$ real vector
bundle $\mathcal{E}$ on $R$. The metric $\mu$ on $\mathcal{T}{\bf
X}$ induces a metric on the bundle $\mathcal{E}$.

 The restriction of the quotient map $\xi_v|_{R_v}: R_v \to X_v $  is
a homeomorphism onto its image. As a result the space $R $ is
homeomorphic to the subspace $\iota(P)$ of $X$ used by Kustarev
\cite{[Kus]}. The map  $\iota: P \to X$ is a homeomorphism given
by the composition
   $P \stackrel{i}{\rightarrow} P\times T_N \stackrel{j}{\rightarrow} X $
    where $i$ is the inclusion
   given by $i(p_1, \ldots, p_n) = (p_1, \ldots, p_n, 1,\ldots,1)$
    and $j$ is the quotient map that
   defines $X$.
 For any face $F$ of $P$ we denote its image in $R$
   under the composition of above homeomorphisms as $R(F)$. The
   restriction of this homeomorphism to the relative interior of
   $F$ is smooth, and we denote the image by $R(F^{\circ})$.

   Let  $\widetilde{X}_{v}(F) $ be the preimage of  $X(F)$ in $\tilde{X}_v$.
   If $F$ is the intersection of facets $F_{i_1},\ldots, F_{i_t}$,
   then  $\widetilde{X}_{v}(F)  $ is the submanifold of
   $\widetilde{X}_v$ defined by the equations $z_{{i_j},v} =0$, $ 1\le j\le
   t$. Then arguments similar to the case of $\mathcal{E}$ show
   that the restrictions $\mathcal{T}\widetilde{X}_{v}(F)|_{R_v \cap R(F)}$
   glue together to produce a subbundle $\mathcal{E}_F$ of
   $\mathcal{E}|_{R(F)}$.

 It  is easy to check from \eqref{chcoord3} that
 \begin{equation}\frac{\partial}{\partial z_{{i_j},v_1}} \left|_{x}  =
 \frac{\partial}{\partial z_{{i_j},v_2}} \right|_x \end{equation} at any
 point $x$ in $R_{v_1}\cap R_{v_2} \cap R(F)$. Therefore we obtain a
 subbundle $\mathcal{N}_F$ of $\mathcal{E} |_{R(F)}$ corresponding
 to the normal bundles of $\widetilde{X}_{F,v} $ in
 $\widetilde{X}_v$. The bundle $\mathcal{N}_F$ obviously splits
 into the direct sum of the rank $2$ bundles $\mathcal{N}_{F_{k}} $
 where $k \in \mathcal{I}(F):=\{i_1, \ldots, i_t\} $.

Recall the torus $T_{N(F)}$ corresponding to the face $F$ of $P$
from equation \eqref{TM} and Definition \ref{TM}. For any vertex
$v$ of $F$, the module $N(F)$ is a direct summand of the module
$N(v)$. Consequently, $T_{N(F)}$ injects into $T_{N(v)}$. Suppose
$x$ is a point in $R(F^{\circ})$.  Then $T_{N(F)}$ is the
stabilizer of any preimage of $x$ in $\widetilde{X}_v$.

$T_{N(F)}$ is the product of the circles $T_{\lambda_{k}}$, $k\in
\mathcal{I}(F)$.    The circle $T_{\lambda_{k}}$
acts nontrivially on $\mathcal{N}_{F_k}$ and induces an almost
complex structure on it corresponding to rotation by
$\frac{\pi}{2}$. Note that this structure depends on the sign of
$\lambda_{k}$ or, in other words, the specific omniorientation.
Thus the $T_{N(F)}$ action induces an almost complex structure on
$\mathcal{N}_{F}$.

Using the method of Kustarev \cite{[Kus]} it is possible to
construct an orthogonal almost complex structure $J$ on
$\mathcal{E}$ that satisfies the following condition:
  ($\star$) For any face $F$  of $ P$  of dimension less than
 $n$,   the restriction of $J$ to $ \mathcal{N}_{F}|_{R(F^{\circ})}$
  agrees with the complex structure induced by the $ T_{N(F)}$  action and
 the omniorientation.

 For future use, we give a brief outline of the proof of existence
of such a structure. The details may be found in \cite{[Kus]}. In
our case, the bundles $\mathcal{E}_F$ and $\mathcal{N}_{F_k}$ play
the roles of the bundles $\tau(M_F)$ and $\xi_k$ in \cite{[Kus]}.

 An orthogonal almost complex structure on $\mathcal{E}$ may  be
regarded as a map $J: R \to SO(2n)/U(n)$. We proceed by induction.
Let $sk_i(R)$ denote the union of all $i$-dimensional faces of
$R$. For $i=0$, existence of $J$ is trivial. Extension to
$sk_1(R)$ is possible due to positivity of omniorientation. For
$i\ge 2$, suppose $J$ is a structure on $sk_{i-1}(R)$ satisfying
the condition ($\star$). Then $J$ may be regarded as a map from
$sk_{i-1}(R)$ to $SO(2i-2)/ U(i-1)$ as it is fixed in the normal
directions by the torus action. Construct a cellular cochain
$\sigma^i_J \in C^i(R, \pi_{i-1}(SO(2i )/U(i))$ by defining the
value of $\sigma^i_J $ on an $i$-dimensional face of $R$ to be the
homotopy class of the value of J on the boundary of the face,
composed with a canonical isomorphism between $\pi_{i-1}(SO(2i-2
)/U(i-1))$ and $\pi_{i-1}(SO(2i )/U(i)) $. $J$ extends to
$sk_i(R)$ if and only if $\sigma^i_J = 0$. Following \cite{[Kus]},
one proves that $\sigma^i_J $ is a cocycle. Therefore, by
contractibility of $R$ it is a coboundary. Suppose $\sigma^i_J =
\delta \beta$, where $\beta \in C^{i-1}(R, \pi_{i-1}(SO(2i )/U(i))
$. Note that $\delta \beta(Q) = \pm \sum_{G \subset
\partial Q} \beta(G)$. For each $H \in sk_{i-1}(R)$, one perturbs
$J$ in the interior of $H$ by a factor of $-\beta(H)$. This makes
$\sigma^i_J= 0$. (Note that if $\beta(H)=0$, no change is required
for face $H$. This will be used crucially in Lemma \ref{extn}.)

By $(\star)$ the structure $J$ on $\mathcal{E}_v$ is
invariant under the action of isotropy groups. We can therefore
use the action of $T_{N(v)}$ to produce an invariant almost
complex structure on $\mathcal{T}\widetilde{X}_v$ as follows,
\begin{equation}
J(t\cdot x) = dt \circ J(x) \circ dt^{-1} \; \forall x \in R_v, \;
{\rm and} \; \forall t\in T_{N(v)}
\end{equation}
 The local group $G_v$ of orbifold chart $(\widetilde{X}_v, G_v, \xi_v)$
 is a subgroup of $T_{N(v)}$. Thus $J$ is $G_v$-invariant on
 $\widetilde{X}_v$.

  The compatibility of $J$ across charts may be
 verified as follows.
  Take any point $x \in X_{v_1} \cap X_{v_2}$. Let
  $\widetilde{x} \in \widetilde{X}_{v_1}$ be a preimage of $x$
  under $\xi_{v_1}$. Suppose $\widetilde{x}= t_1\cdot x_0$ where
  $x_0 \in R$ and $t_1 \in T_{N(v_1)}$. Choose an embedding
  $\widetilde{\phi}_{12}$ of a small
  $G_x$-stable neighborhood of $\widetilde{x}$ into
  $\widetilde{X}_{v_2}$ as outlined in section \ref{comp}. Suppose
   $\widetilde{\phi}_{12}({\widetilde{x}}) = t_2 \cdot
  x_0$ where $t_2 \in T_{N(v_2)}$.
 Then
\begin{equation}
\widetilde{\phi}_{12} = t_2 \circ \phi_{12} \circ {t_1}^{-1}
\end{equation}
By construction of $J$ on $\mathcal{E}$, $J$ commutes with
$d\phi_{12}|_R$. $J$ commutes with $dt_i$ and $dt_i^{-1} $ by its
construction on $\widetilde{X}_{v_i}$. Therefore $J$ commutes with
$d\widetilde{\phi}_{12}$, as desired.
\end{proof}

\begin{theorem}\label{extn} Suppose an orthogonal invariant
  almost complex structure is given on a characteristic suborbifold
 ${\bf X}(F)$. Then it can be extended to ${\bf X}$.
 \end{theorem}

\begin{proof} We follow the notation of the previous theorem.
  $J$ has been already specified on ${\bf X}(F)$ where
$\dim(F) = n-1$. This determines $J$ on the subbundle
$\mathcal{E}_F$ of $\mathcal{E}$ over $R(F)$. We use the torus
action and omniorientation to extend $J$ to $\mathcal{E}|_{R(F)}$.

 We construct an extension of $J$ to $R$ skeleton-wise. Extension
 up to $sk_1(R)\cup F$ is achieved using positivity of
 omniorientation. For extension to higher skeletons we need to use
 obstruction theory.  We need to take care so that $J$ is preserved on
 sub-faces of $F$. We use induction. Suppose $J$ has been extended
 to $sk_{d-1}(R) \cup F$, where $d < n $. (We will deal with the $d=n$
 case separately.)

 Let $\sigma^d \in
C^d(R, \pi_{d-1} (SO(2d)/U(d)) )$ be the obstruction
 cocycle.
   Let $i: R(F) \hookrightarrow R$ be inclusion map.
 Restriction to $F$ produces a cochain $$ i^{*}(\sigma^d) \in C^d(R(F),
 \pi_{d-1} (SO(2d)/U(d)) ).$$
 Then $ i^{*}(\sigma^d)=0$ since we know that $J$ extends to $R(F)$.
   Since $\sigma^d = \delta \beta $,
  $ i^{*}(\beta)$ is a cocycle.  As $R(F)$ is contractible $ i^{*}(\beta)$
 is a coboundary. Let $ i^{*}(\beta)=\delta \beta_1$ where $\beta_1 \in  C^{d-2}(R(F))$.
  Define a chain $\beta_2 \in C^{d-2}(R)$  such that
\begin{equation}
   \beta_2(H) = \left\{ \begin{array}{ll}
\beta_1(H) & {\rm for\, any\,}  (d-2)
 \, {\rm face} \, H \subset R(F)\\  0 & {\rm otherwise}
 \end{array} \right.
\end{equation}

  Then define
 $\beta_3= \beta - \delta (\beta_2)$. This new cochain has the property
 that  $ \delta(\beta_3) = \sigma^d$ and its
 action $(d-1)$-dimensional faces of $R(F)$ is zero. So we can now
 extend the structure to $sk_d \cup R(F)$ without affecting the sub-faces of
 $R(F)$.

 By induction, we may assume that $J$ has been extended to $sk_{n-1}(R) \cup R(F)$.
 Let $\sigma^n \in C^n(R, \pi_{n-1} (SO(2n)/U(n) )$  be the corresponding obstruction
  cochain for extension to $sk_n$. Since $R$ is contractible we
  have $\sigma^n= \delta \beta$. We modify $\beta $ as follows.
  Suppose $K$ is a facet adjacent to $F$. Define $\beta^{\prime}
  \in C^{n-1}$ as follows.
  \begin{equation} \beta^{\prime}(H)= \left\{ \begin{array}{ll}
  0 & {\rm if\,} H=R(F)\\
  \beta(R(F)) + \beta(R(K)) & {\rm if\,} H=R(K)\\
  \beta(H) & {\rm otherwise} \end{array} \right.
  \end{equation}
  Then $\delta \beta^{\prime} =\delta \beta= \sigma^n$ and $\beta^{\prime}(R(F))=0$.
  So we may extend $J$ to $R$ without changing it on $R(F)$.
\end{proof}

\begin{corollary}\label{extc} Suppose an orthogonal invariant
  almost complex structure is given on a suborbifold
 ${\bf X}(F)$ where $F$ is any face of $P$.
 Then it can be extended to ${\bf X}$.
\end{corollary}

\begin{proof} Consider a nested sequence of faces $F=H_0 \subset
H_1 \ldots \subset H_k= P$ where $\dim(H_i)= \dim(F) +i$. Extend
the structure inductively from ${\bf X}(H_i)$ to ${\bf
X}(H_{i+1})$ using Theorem \ref{extn}.
\end{proof}

\section{Blowdowns}\label{blowdown}

 Topologically the blowup will correspond to replacing an
 invariant suborbifold by the projectivization of its normal bundle. Combinatorially
 we replace a face by a facet with a new characteristic vector. Suppose $F$ is a face of $P$. We
 choose  a hyperplane $H = \{ \widehat{p}_0 = 0 \}$ such that $\widehat{p}_0$
 is negative on $F$ and $\widehat{P}:=\{\widehat{p}_0 > 0\} \cap P$ is a simple
 polytope having one more facet than $P$. Suppose $F_1, \ldots, F_m$ are the
 facets of $P$. Denote the facets $F_i \cap \widehat{P} $ by $F_i$
 without confusion. Denote the extra facet $H \cap P$ by
 $F_{0}$.

Without loss of generality let $F = \bigcap_{j= 1}^k F_j$.
  Suppose there exists a primitive vector
  $\lambda_{0} \in N$ such that
  \begin{equation}
 \lambda_{0} = \sum_{j= 1}^k b_j \lambda_j, \; b_j > 0 \, \forall
 \, j.
  \end{equation}
 Then the assignment $F_{0} \mapsto \lambda_{0}$
 extends the characteristic function of $P$ to a characteristic function
 $\widehat{\Lambda}$ on $\widehat{P}$. Denote the omnioriented quasitoric
 orbifold derived from the model $(\widehat{P}, \widehat{\Lambda}  )$ by
 ${\bf Y}$.

Consider a small open neighborhood $U:=\{ x \in P:
\widehat{p}_0(x) < \epsilon\} $ of the face $F$, where $0 < \epsilon <1$.
 Denote $U \cap \widehat{P}$ by $\widehat{U}$.
 By Lemma \ref{class} we may assume that
 \begin{equation} f: U = F \times [0,1)^{k} \end{equation}
We also assume without loss of generality that the defining
function $\widehat{p}_{j}$ of the facet $F_{j}$ equals the $j$-th
coordinate $p_j$ of $\RR^n$ on $U$, for each $1\le j \le k$.

  Choose small positive numbers $\epsilon_1 < \epsilon_2< \epsilon $ and a
   smooth non-decreasing function $\delta: [0,\infty) \to \RR$
    such that
   \begin{equation}
   \delta(t)= \left\{ \begin{array}{ll} t & {\rm if} \, t < \epsilon_1 \\
   1 & {\rm if} \, t > \epsilon_2 \end{array} \right.
   \end{equation}

 Then define $\tau:  \widehat{P} \to  P  $ to be the map given by
\begin{equation}
 \tau( p_1,\ldots, p_k, p_{k+1},\ldots, p_n ) = ( \delta(\widehat{p}_0)^{ b_1} p_1, \ldots, \delta(\widehat{p}_0 )^{ b_k}
 p_k,p_{k+1}, \ldots, p_n ).
\end{equation}
   The  blow down map $\rho: ( \widehat{P} \times T_N/\sim) \to (P \times T_N/\sim) $
    is defined by
   \begin{equation}\label{rho}
   \rho({\bf p}, {\bf q}) = ( \tau({\bf p}), {\bf q}).
   \end{equation}

Since $\delta=1$ if $\widehat{p}_0 > \epsilon_2$, $\rho$ is a
diffeomorphism of orbifolds away from a tubular neighborhood of
$X(F)$. We study the map $\rho$ near $X(F)$.

 Let $w= \bigcap_{j=1}^n F_j$  be a vertex of $F$. Suppose $v$ be a vertex of $F_{0}$
 such that $\tau(v)=w$. Then the edge joining $v$ and $w$ is the
 intersection of $n-1$ facets common to both which must include
 $F_{k+1}, \ldots, F_n $. Therefore there are $k$ choices for
 $v$, namely $v_i =   \bigcap_{0\le j\neq i \le n} F_j $ with
 $1 \le i \le k$.

 Let $\widehat{p}_j = 0$ be the defining equation of the facet
 $F_j$ for $k+1 \le j\le n$. Order the facets at $w$ as $F_1,
 \ldots,
 F_n$, and those at $v_i$ as $F_1, \ldots, F_{i-1}, F_0, F_{i+1}, \ldots, F_n
 $. Let $z_{j,w}$ and $z_{j,v_i}$ be the coordinates on
 $\widetilde{X}_w$ and $\widetilde{Y}_{v_i}$ defined according to
 \eqref{homeo} and \eqref{cplxcoor}. Then by using a process
 similar to the one used for \eqref{chcoord3},
% \eqref{rho} and \eqref{chcoord2}
  we obtain the following description of
 $\rho$ near $Y_{v_i}$,
\begin{equation}\label{rhoinc} \begin{array}{ll}
z_{i,w} \circ \rho = z_{i,v_i}^{b_i} \sqrt{p_i \delta(\widehat{p}_0)^{ b_i} (\widehat{p}_0)^{-b_i}} & \\
 z_{j,w} \circ \rho = z_{i,v_i}^{b_j} z_{j,v_i} \sqrt{ \delta(\widehat{p}_0)^{ b_j} (\widehat{p}_0)^{-b_j} } & {\rm if} \; 1\le j
 \neq i \le k \\
 z_{j,w} \circ \rho = z_{j,v_i} & {\rm if }\; k+1 \le j \le n
 \end{array}
\end{equation}

We define a new coordinate system on $\widetilde{Y}_{v_i}$, for
each $1\le i \le k$, as follows.
\begin{equation}\label{newcoord} \begin{array}{ll}
z_{i,v_i}^{\prime} = z_{i,v_i} (\sqrt{ p_i} )^{1/b_i} \sqrt{ \delta(\widehat{p}_0) (\widehat{p}_0)^{-1}} &  \\
z_{j,v_i}^{\prime} = z_{j,v_i} (\sqrt{ p_i} )^{-b_j/b_i} & {\rm if} \; 1\le j \neq i \le k \\
z_{j,v_i}^{\prime} = z_{j,v_i} & {\rm if }\; k+1 \le j \le n
\end{array}
\end{equation}
This is a valid change of coordinates as $p_i$ is positive on
$\widetilde{Y}_{v_i}$ and $\delta(\widehat{p}_0)
(\widehat{p}_0)^{-1} $ is identically one near $\widehat{p}_0 =
0$.
% Note that this is equivalent to a change of coordinates on
% each cone of $\widehat{P}$

 In these new coordinates, $\rho$ can be expressed as
\begin{equation}\label{rhoinc2} \begin{array}{ll}
z_{i,w} \circ \rho = (z_{i,v_i}^{\prime})^{b_i}  & \\
 z_{j,w} \circ \rho = (z_{i,v_i}^{\prime})^{b_j} z_{j,v_i}^{\prime}  & {\rm if} \; 1\le j
 \neq i \le k \\
 z_{j,w} \circ \rho = z_{j,v_i}^{\prime} & {\rm if }\; k+1 \le j \le n
 \end{array}
\end{equation}

\begin{lemma} The restriction  $\rho: {\bf Y}-{\bf Y}(F_0) \to {\bf
X}-{\bf X}(F) $ is a diffeomorphism of orbifolds.
\end{lemma}

\begin{proof} This is obvious outside $\pi^{-1}(U)$. On $\pi^{-1}(U)-X(F)$,
 by formula \eqref{rhoinc2}, $\rho$ is locally equivalent to a blowup in
complex geometry. Therefore $\rho$ is an analytic isomorphism on  $\pi^{-1}(U)-X(F)$.
However since our quasitoric orbifolds are primitive, there is no complex reflection
in our orbifold groups. Hence using the results of \cite{[Pri]},  analytic isomorphism
yields diffeomorphism of orbifolds.
\end{proof}

\begin{lemma} If ${\bf X}$ is
positively omnioriented, then so is a blowup ${\bf Y}$.
\end{lemma}

\begin{proof} Recall the positive ordering of facets at a vertex $v$ in
section \ref{omnio} to define the matrix $\Lambda_{(v)}$ whose
determinant has the same sign as sign of $v$.

Let $w$ be any vertex of $F$ and $v_i$ be any vertex in
$\rho^{-1}(w)$. Let $F_1, \ldots, F_n$ be  positively ordered
facets at $w$.
 An inward normal vector to $F_0$ is a positive linear combination of the
inward normal vectors to $F_1, \ldots, F_k$. Therefore $F_1,
\ldots, F_{i-1}, F_0, F_{i+1} \ldots$, $F_n$ are  positively
ordered for each $i= 1, \ldots,k$.
 So the matrix $\Lambda_{(v_i)}$ is obtained by replacing
the $i$-th column of $\Lambda_{(w)}$, namely $\lambda_i$, by
$\lambda_0 = \sum_{j=1}^k b_j \lambda_j$. Therefore $\det
\Lambda_{(v_i)} = b_i \det \Lambda_{(w)}$. The lemma follows.
\end{proof}

\begin{defn}\label{resol} A blowdown $\rho$ is said to be a resolution if for
any vertex $w$ of the exceptional face $F$ and any vertex $v_i \in
\rho^{-1}(F)$ we have $o(G_{v_i}) < o(G_w)$.
\end{defn}

\begin{lemma} A blowdown $\rho$ is a resolution if $b_i< 1$ for each $i$.
\end{lemma}

\begin{proof}
The lemma holds since by \eqref{ordgv} we have $o(G_{v_i})= |\det
\Lambda_{v_i}| = b_i |\det \Lambda_{w}|= b_i o(G_w) $.
\end{proof}

\section{Pseudoholomorphic  blowdowns}\label{pseudo}

\begin{lemma}\label{pseubd}
Let $\rho: Y \to X $ be a blowdown along  a subset $X(F)$. Suppose
there exist holomorphic coordinate systems $z_{1,w}^{\ast},
\ldots, z_{n,w}^{\ast} $ on the uniformizing chart
$\widetilde{X}_w$ for every vertex $w$ of $F$, which produce an
analytic structure on a neighborhood $\pi^{-1}(U)$ of $X(F)$.
Assume further that this analytic structure extends to an almost
complex structure on ${\bf X}$. Then the blowup induces an almost
complex structure on ${\bf Y}$ which is analytic near the
exceptional set $Y(F_0)$. Moreover, with respect to these
structures $\rho $ is analytic near $Y(F_0)$ and an almost complex
diffeomorphism of orbifolds away from $Y(F_0)$.
\end{lemma}

\begin{proof} Note that for two vertices $w_1$, $w_2$ of $F$,
the coordinates must be related as
\begin{equation}\label{w1w2}
z_{j,w_2}^{\ast} = \prod_{i=1}^n (z_{i,w_1}^{\ast})^{d_{ij}}
\end{equation}
where the $d_{ij}$s are  rational numbers determined from the
matrix $\Lambda_{w_2}^{-1} \Lambda_{w_1}$, see
\eqref{chcoord2} and \eqref{chcoord3}.

 Also the coordinates $z_{j,w}^{\ast}$ have to relate to the
 coordinates defined in \eqref{homeo} and \eqref{cplxcoor} as
 follows,
 \begin{equation}\label{zstar} z_{j,w}^{\ast} = z_{j,w} f_j, \; 1\le j \le n \end{equation}
  where each $f_j$ is smooth and non-vanishing on
  $\widetilde{X}_w$.
  For each $v_i \in \rho^{-1}(w)$ we define coordinates in
  its neighborhood, by modifying the coordinates of
  \eqref{newcoord} as follows,
  \begin{equation}\label{newcoord2} \begin{array}{ll}
z_{i,v_i}^{\ast} = z_{i,v_i}^{\prime} (f_i\circ \tau)^{1/b_i} &  \\
z_{j,v_i}^{\ast} = z_{j,v_i}^{\prime} (f_j \circ \tau)
(f_i\circ \tau)^{-b_j/b_i}   & {\rm if} \; 1\le j \neq i \le k \\
z_{j,v_i}^{\ast} = z_{j,v_i}^{\prime} & {\rm if }\; k+1 \le j
\le n
\end{array}
\end{equation}
In these coordinates $\rho$ takes the following form near $v_i$,
\begin{equation}\label{rhoinc3} \begin{array}{ll}
z_{i,w}^{\ast} \circ \rho = (z_{i,v_i}^{\ast})^{b_i}  & \\
 z_{j,w}^{\ast} \circ \rho = (z_{i,v_i}^{\ast})^{b_j} z_{j,v_i}^{\ast}  & {\rm if} \; 1\le j
 \neq i \le k \\
 z_{j,w}^{\ast} \circ \rho = z_{j,v_i}^{\ast} & {\rm if }\; k+1 \le j \le n
 \end{array}
\end{equation}
We define an almost complex structure $\widehat{J}$ on $\bf{Y}$ by
defining the coordinates $z_{j,v_i}^{\ast}$ to be holomorphic
near
 $Y(F)$  and  by $\widehat{J} = d\rho^{-1} \circ J \circ d\rho
$ away from it. This is consistent as $\rho$ is a diffeomorphism
of orbifolds on the complement of $Y_F$.

  By \eqref{w1w2} and \eqref{rhoinc3}, for any two vertices $u_1$
  and $u_2$ of $F_0$, we have
\begin{equation} z_{j,u_2}^{\ast} = \prod_{i=1}^n (z_{i,u_1}^{\ast})^{e_{ij}}
\end{equation}
for some rational numbers $e_{ij}$. But these numbers are
determined by the matrix $\Lambda_{u_2}^{-1} \Lambda_{u_1}$.
It is then obvious from the arguments about compatibility of
charts in section \ref{diffs} that the patching of the charts
$Y_{u_1}$ and $Y_{u_2}$ is holomorphic.
\end{proof}

Examples of blowdowns that satisfy the hypothesis of Lemma
\ref{pseubd} include blowdowns of four dimensional positively
 omnioriented quasitoric orbifolds constructed in \cite{[GP]}
  and toric blow-ups of simplicial toric varieties.

\begin{defn}\label{smfn}\cite{[GP]}
A function $f$ on $X$ is said to be smooth if
$f \circ \xi $ is smooth for every uniformizing system $(\widetilde{U},G,\xi)$.
 A complex valued smooth function $f$
 on an almost complex
orbifold $({\bf X}, J)$ is said to be $J$-holomorphic if the
differential $d(f \circ \xi)$ commutes with $J$ for every chart
$(\widetilde{U},G,\xi)$. We denote the sheaf of $J$-holomorphic
functions on ${\bf X}$ by $\Omega^0_{J,X}$. A continuous map
$\rho: Y \to X $ between almost complex orbifolds $({\bf Y}, J_2)$
and $({\bf X}, J_1) $ is said to be pseudo-holomorphic if $f \circ
\rho \in \Omega^0_{J_2,Y}(\rho^{-1}(U))$ for every $f \in
\Omega^0_{J_1,X}(U) $ for any open set $U  \subset X$; that is,
$\rho$ pulls back pseudo-holomorphic functions to
pseudo-holomorphic functions.
\end{defn}

\begin{lemma}\label{pseubd2} Blowdowns
 that satisfy the hypothesis of
lemma \ref{pseubd} are pseudoholomorphic.
\end{lemma}

\begin{proof} Suppose $\rho: Y \to X$ is such a blowdown.
Since $\rho$ is an almost complex diffeomorphism of orbifolds away
from the exceptional set $Y(F_0)$, it suffices to check the
statement near $Y(F_0)$. Pick any vertex $w$ of $F$. Define $W=
X_w \cap \pi^{-1}(U)$. For any vertex $v_i \in \rho^{-1}(w)$, let
$V_i = Y_{v_i} \cap \rho^{-1}(\pi^{-1}(U))$. We will denote the
characteristic vectors at $v_i$ by $\widehat{\lambda}_j, \, j=1,
\ldots,n$. Note that
\begin{equation}
\widehat{\lambda}_j = \left\{ \begin{array}{ll} \lambda_j & {\rm
if}\, j\neq i \\ \lambda_0 & {\rm if}\, j=i. \end{array} \right.
\end{equation}

The ring $\Omega^0_{J_1, X}(W)$ is the $G_{w}$-invariant subring
of convergent power series in variables $z_{j,w}^{\ast}$. It is
generated by monomials of the form
\begin{equation}\label{f}
f= \prod_{j=1}^n (z_{j,w}^{\ast})^{d_j}
\end{equation}
where the $d_j$s are integers such that $\sum a_j d_j$ is an integer
whenever the vector $\sum a_j \lambda_j \in N$. This last condition
follows from invariance under action of the element $g\in G_w $ corresponding
to $\sum a_j \lambda_j$.

Using \eqref{rhoinc3} and $\lambda_0 = \sum_{j=1}^n b_j \lambda_j$ with $b_j=0$
for $j \ge k+1$, we get
\begin{equation}
f \circ \rho = (z_{i,v_i}^{\ast})^{\sum b_j d_j} \, \prod_{j\neq
i} (z_{j,v_i}^{\ast})^{d_j}.
\end{equation}

Take any element $h $ in $G_{v_i}$. Suppose $h$ is represented by $\sum c_j \widehat{\lambda}_j \in N$.
The action of $h$ on $f \circ \rho$ is multiplication by $e^{2\pi\sqrt{-1}\alpha}$, where
\begin{equation}
\alpha = c_i \sum_j b_j d_j + \sum_{j\neq i}c_j d_j
 = c_i b_i d_i + \sum_{j\neq i} (c_j + c_i b_j) d_j.
 \end{equation}

 Note that $ \eta := c_i b_i \lambda_i + \sum_{j\neq i} (c_j + c_i b_j)
\lambda_j = c_i \sum_j b_j \lambda_j + \sum_{j\neq i} c_j \lambda_j = \sum c_j \widehat{\lambda}_j$.
Hence this is an element of $N$.

 Suppose $f$ is a generator of $\Omega^0_{J_1, X}(W)$ as in \eqref{f}. Consider the action of the element of $G_w$
 corresponding to $\eta$ on $f$. It is multiplication by $e^{2\pi\sqrt{-1}\alpha}$.
 Since $f$ is $G_w$-invariant,
$\alpha $ is an integer. Hence $f\circ \rho$ is $G_{v_i}$
invariant. The ring $\Omega^0_{J_1, Y}(V_i)$ is the
$G_{v_i}$-invariant subring of convergent power series in
variables $z_{j,v_i}^{\ast}$. Therefore $f\circ \rho \in
\Omega^0_{J_1, Y}(V_i)$.
\end{proof}

The proof of the following corollary of Lemma \ref{pseubd} is
straightforward.

\begin{corollary}\label{seq}
 Consider a sequence of blowups $\rho_i: Y_i \to Y_{i -1} $ where
 $1\le i \le r$ and $\rho_1$ satisfies the hypothesis of Lemma
 \ref{pseubd}.
 Assume that the locus of the $i$-th blowup is  contained in the
exceptional set of the $(i-1)$-st blowup for every $i$. Then we can inductively
choose almost complex structures so that each blowdown map in the
sequence is pseudoholomorphic.
\end{corollary}

\begin{theorem}\label{thmresol}  There exists a pseudoholomorphic resolution of
  singularity for any primitive positively omnioriented four dimensional
   quasitoric orbifold.
\end{theorem}

\begin{proof} For any primitive positively omnioriented four dimensional
   quasitoric orbifold, Theorem 3.1 of \cite{[GP]} produces an almost complex
   structure that satisfies the hypothesis of Lemma \ref{pseubd} for every
   vertex. The singularities are all cyclic. We can resolve them by applying
   a sequence of blow-ups as in Corollary \ref{seq}.
\end{proof}

\section{Crepant blowdowns}\label{crepbd}

\begin{defn}\label{crepant}
A blowdown is called crepant if $\sum b_j = 1 $.
\end{defn}

This has the following geometric interpretation.

 \begin{defn}
Given an almost complex $2n$-dimensional orbifold $({\bf X}, J)$,
we define the
 canonical sheaf $K_{X}$ to be the sheaf of continuous $(n,0)$-forms on $X$;
  that is, for any orbifold chart
$(\widetilde{U},G,\xi)$ over an open set $U \subset X$,
 $K_X (U) = \Gamma( \wedge^n \mathcal{T}^{1,0}(\widetilde{U})^{\ast})^G$ where $\Gamma$ is the functor
 that takes continuous sections.
\end{defn}

 An almost complex orbifold is called
Gorenstein or $SL$ orbifold if the linearization of every local
group element $g$ belongs to $SL(n,\CC)$.  For an $SL$-orbifold
${\bf X}$, the canonical sheaf is a complex line bundle over $X$.

\begin{lemma} Suppose $\rho: Y \to X $ is a  pseudoholomorphic
 blowdown of $SL$ quasitoric orbifolds along a face $F$
 satisfying the hypothesis of Lemma \ref{pseubd}.
  Then $\rho$ is crepant if and only if
 $\rho^{\ast} K_X = K_Y$.
\end{lemma}

\begin{proof}  We consider the canonical sheaf $K_X$ as a sheaf of modules
 over the sheaf of continuous functions $\mathcal{C}^0_X$.
  Since $\rho$ is an almost complex diffeomorphism away
from the exceptional set it suffices to check the equality of the
$\rho^{\ast} K_Y$ and $K_X$ on the neighborhood
$\rho^{-1}(\pi^{-1} (U)) \subset Y$ of the exceptional set. Choose
any vertex $w$ of F. On $X_w \cap \pi^{-1} (U)$, the sheaf $K_X$
is generated over the sheaf $\mathcal{C}^0_X$ by the form
$dz_{1,w}^{\ast} \wedge \ldots \wedge dz_{n,w}^{\ast} $, see
\eqref{zstar}. Let $v_i$ be any preimage of $w$ under $\rho$.
Similarly on $Y_{v_i} \cap \rho^{-1}(\pi^{-1} (U)) $, $K_Y$ is
generated over the sheaf $\mathcal{C}^0_Y$ by the form
$dz_{1,v_i}^{\ast} \wedge \ldots \wedge dz_{n,v_i}^{\ast} $.

 Using
\eqref{rhoinc3} we have
\begin{equation}\label{drhoinc3} \begin{array}{ll}
\rho^{\ast} dz_{i,w}^{\ast} =  b_i (z_{i,v_i}^{\ast})^{b_i-1}  dz_{i,v_i}^{\ast}  & \\
 \rho^{\ast} dz_{j,w}^{\ast}  =  (z_{i,v_i}^{\ast})^{b_j}  dz_{j,v_i}^{\ast} +
  b_j (z_{i,v_i}^{\ast})^{b_j-1} z_{j,v_i}^{\ast}  dz_{i,v_i}^{\ast}    & {\rm if} \; 1\le j
 \neq i \le k \\
 \rho^{\ast} dz_{j,w}^{\ast}  = dz_{j,v_i}^{\ast} & {\rm if }\; k+1 \le j \le
 n.
 \end{array}
\end{equation}
Therefore we have \begin{equation} \rho^{\ast} (dz_{1,w}^{\ast}
\wedge \ldots \wedge dz_{n,w}^{\ast}) = b_i
(z_{i,v_i}^{\ast})^{b_1 + \ldots +b_k -1} dz_{1,v_i}^{\ast} \wedge
\ldots \wedge dz_{n,v_i}^{\ast}.
\end{equation}
The lemma follows.
\end{proof}

\section{Chen-Ruan Cohomology}\label{crc}

The Chen-Ruan cohomology group is built out of the ordinary
cohomology of certain copies of singular strata of an orbifold
called twisted sectors. The twisted sectors of orbifold toric
varieties was computed in \cite{[Po]}. The determination of such
sectors for quasitoric orbifolds is similar in essence. Another
important feature of Chen-Ruan cohomology is the grading which is
 rational in general. In our case the grading will  depend on
 the omniorientation.

 Let ${\bf X}$ be an omnioriented
quasitoric orbifold. Consider any element $g$ of the group $G_F $
\eqref{gx3}. Then $g$ may be represented by a vector $\sum_{j \in
\mathcal{I}(F)} a_j \lambda_j $. We may restrict $a_j$ to $[0,1)\cap \QQ$.
Then the above representation is unique. Then define the degree
shifting number or age of $g$ to be
\begin{equation}\label{age}
\iota(g)= \sum a_j.
\end{equation}

For faces $F$ and $H$ of $P$ we write $F \le H$ if $F$ is a
sub-face of $H$, and $F < H$ if it is a proper sub-face. If $F \le H
$ we have a natural inclusion of $G_H$ into $G_F$ induced by the
inclusion of $N(H)$ into $N(F)$. Therefore we may regard $G_H$ as
a subgroup of $G_F$. Define the set
\begin{equation}
G_F^{\circ} = G_F - \bigcup_{F < H} G_H
\end{equation}
Note that $G_F^{\circ} = \{ \sum_{j \in \mathcal{I}(F)} a_j \lambda_j |
0 < a_j < 1 \} \cap N  $, and $G_P^{\circ}= G_P =\{0\}$.

\begin{defn}\label{CR}
 We define the Chen-Ruan orbifold cohomology
 of an omnioriented quasitoric orbifold ${\bf X}$ to be
$$ H^{\ast}_{CR}({\bf X}, \RR ) =
\bigoplus_{F \le P} \bigoplus_{ g\in G_F^{\circ}} H^{\ast - 2
\iota(g)} (X(F), \RR).$$
 Here $H^{\ast}$ refers to singular cohomology or equivalently to
 de Rham cohomology of invariant forms when $X(F)$ is considered
  as the orbifold ${\bf X}(F)$.
 The pairs $(X(F), g)$ where $F<P$  and $ g\in G_F^{\circ}$ are
called twisted sectors of ${\bf X}$. The pair $(X(P),1)$, i.e. the
underlying space $X$, is called the untwisted sector. We denote
the Betti number rank$(H^d_{CR}({\bf X}))$ by $h^d_{CR}$.
\end{defn}

Note that if ${\bf X}$ is a manifold then its Chen-Ruan cohomology
is same as its singular cohomology.

\subsection{Poincar\'e duality}\label{PD}
Poincar\'e duality is established in a similar fashion as for
compact almost complex orbifolds.
 We need to distinguish the copies of
$X(F)$ corresponding to different twisted sectors. Therefore for
$g \in G_F^{\circ}$, we define the space
\begin{equation}
S(F,g) = \{(x,g): x \in X(F) \}.
\end{equation}
Of course $S(F,g)$ is homeomorphic to $X(F)$. It is denoted by
${\bf S}(F,g)$ when endowed with an orbifold structure which is
the structure of ${\bf X}(F)$ with an additional trivial action of
$G_F$ at each point. With this structure, it is a suborbifold of
${\bf X}$ in a natural way.
 The untwisted sector is denoted by $S(P,1)$. In this
notation the Chen-Ruan groups may be written as
\begin{equation}
H^{\ast}_{CR}({\bf X}, \RR ) = \bigoplus_{F \le P} \bigoplus_{
g\in G_F^{\circ}} H^{\ast - 2 \iota(g)} (S(F,g), \RR)
\end{equation}

\begin{lemma}\label{codim}
Suppose $ g\in G_F^{\circ}$. Then $2\iota(g) + 2\iota(g^{-1}) = 2n
- \dim(X(F)) $.
\end{lemma}

\begin{proof} When $F=P$, $G_P^{\circ}= \{0\}$ and the result
is obvious. Suppose $F= \bigcap_{i=1}^k F_i $. Then
$g=\sum_{i=1}^k a_i \lambda_i$ where each $0< a_i <1$. Then
$g^{-1}$ is represented by the vector $\sum_{i=1}^k -a_i \lambda_i$
in $N$ modulo $N(F)$. Therefore $g^{-1}$ may be identified with
the vector $\sum_{i=1}^k (1-a_i)\lambda_i$. Note that $ 0 < 1 -
a_i<1$ for each $i$. Therefore the age of $g^{-1}$,
$\iota(g^{-1})= \sum_{i=1}^k (1-a_i) $. Hence $2\iota(g) +
2\iota(g^{-1}) = 2 \sum_{i=1}^k a_i + 2\sum_{i=1}^k (1-a_i) = 2k =
2n - \dim(X(F)) $.
\end{proof}

 For any compact orientable orbifold, there exists a notion of
orbifold integration $\int^{orb}$ for invariant top dimensional
forms which gives Poincar\'e duality for the de Rham cohomology of
the orbifold, see \cite{[CR]}. For a chart ${\bf U}=(\widetilde{U}, G, \xi)$
orbifold integration for an invariant form $\omega$ on $\widetilde{U}$ is
defined by
\begin{equation}
\int_{{\bf U}}^{orb} \omega = \frac{1}{o(G)} \int_{\widetilde{U}} \omega.
\end{equation}

 Let $I: {\bf S}(F,g) \to {\bf
S}(F,g^{-1})$ be the diffeomorphism of orbifolds defined by
$I(x,g)= (x, g^{-1})$. We define a bilinear pairing
\begin{equation}\label{pairing1}
\langle , \rangle^{orb}_{(F,g)}: H^{d - 2\iota(g)} (S(F,g)) \times
H^{2n -d - 2\iota(g^{-1})}(S(F,g^{-1})) \longrightarrow \RR
\end{equation}
for every $0\le d \le 2n$ by
\begin{equation}\label{pairing2}
\langle \alpha,\beta \rangle^{orb}_{(F,g)}= \int_{{\bf
S}(F,g)}^{orb} \alpha \wedge I^{\ast}(\beta).
\end{equation}
This pairing is nondegenerate because of Lemma \ref{codim}.
 By taking a direct sum of the pairing \eqref{pairing1} over
 all pairs of sectors $((F,g), (F,g^{-1}))$ for $F \le P$, we
 get a nonsingular pairing for each $0\le d \le 2n$
\begin{equation}\label{pairing3}
\langle , \rangle^{orb}: H^{d}_{CR} ({\bf X}) \times H^{2n
-d}_{CR}({\bf X}) \longrightarrow \RR.
\end{equation}

\section{McKay correspondence}

First we introduce some notation. Consider a codimension $k$ face
$F= F_1 \cap \ldots \cap F_k$ of $P$ where $k \ge 1$.
  Define a $k$-dimensional cone $C_F$ in $N\otimes \RR$ as follows,
\begin{equation}\label{cf}
C_F = \{ \sum_{j=1}^k a_j \lambda_j: a_j \ge 0 \}
\end{equation}
 The group $G_F$ can be identified with the subset $Box_F $ of
 $C_F$, where
\begin{equation}\label{boxf}
 Box_F := \{
\sum_{j=1}^k a_j \lambda_j: 0\le a_j <1 \} \cap N.
\end{equation}
Consequently the set $ G_F^{\circ}$ is identified with the subset
\begin{equation}\label{boxfo}
 Box_F^{\circ} :=  \{ \sum_{j=1}^k a_j
\lambda_j: 0 < a_j < 1 \} \cap N \end{equation} of the interior of
$C_F$. We define $Box_P = Box_P^{\circ}= \{0\} $.

 Suppose $v=F_1\cap \ldots \cap F_n$ is a vertex of $P$. Then
$Box_v =  \bigsqcup_{v\le F} Box_{F}^{\circ} $. This implies
\begin{equation}\label{Gdecom}
G_v = \bigsqcup_{v \le F} G_F^{\circ}
\end{equation}

\subsection{Euler characteristic}
An almost complex orbifold is $SL$ if the linearization of each
$g$ is in $SL(n,\CC)$. This is equivalent to $\iota(g)$ being
integral for every twisted sector.
 Therefore, to suit our purposes, we make the following definition.

\begin{defn}\label{quasisl}
 An omnioriented quasitoric
orbifold is said to be quasi-$SL$ if the age of every twisted
sector is an integer.
\end{defn}

\begin{lemma}\label{lemeuler} Suppose ${\bf X}$ is a quasi-$SL$ quasitoric orbifold.
 Then the Chen-Ruan Euler
characteristic of ${\bf X}$ is given by
$$ \chi_{CR} ({\bf X}) =  \sum_{v} o(G_v) $$ where $v$ varies over
all vertices of $P$.
\end{lemma}

\begin{proof}
 Note that each $X(F)$ is a quasitoric orbifold. So its
cohomology is concentrated in even degrees, see \cite{[PS]}. Since
${\bf X}$ is quasi-$SL$, the shifts $2 \iota(g)$ in grading are
also even integers.
 Therefore the
Euler characteristic of Chen-Ruan cohomology is given by
\begin{equation}\label{euler1}
\chi_{CR} ({\bf X}) =  \sum_{F \le P} \chi (X(F))\cdot
o(G_F^{\circ}).
\end{equation}

Each $X(F)$ admits a decomposition into even dimensional strata as
follows
\begin{equation} X(F) = \bigsqcup_{H \le F} X(H^{\circ})
\end{equation}
where $H^{\circ}$ is the relative interior of $H$ and $
X(H^{\circ})= \pi^{-1}(H^{\circ}) $. We have
\begin{equation}\label{euler2}
\chi (X(F))= \sum_{H\le F} \chi (X(H^{\circ}))
\end{equation}
However $X(H^{\circ}) $ is homeomorphic to the product of
$H^{\circ}$ with  $(S^1)^{{\rm dim}(H)} $. Therefore $\chi
(X(H^{\circ}))= 0 $ unless $H$ is a vertex. Hence
\begin{equation}\label{euler3}
\chi (X(F))= {\rm number \; of\; vertices\; of \; } F.
\end{equation}
This formula also follows from the description of the homology
groups of a quasitoric orbifold in \cite{[PS]}.

 Using \eqref{Gdecom}, \eqref{euler1} and \eqref{euler3}, we
have the desired formula for $\chi_{CR} ({\bf X})$.
\end{proof}

\begin{lemma} The crepant blowup of a quasi-$SL$ quasitoric orbifold is
quasi-$SL$.
\end{lemma}

\begin{proof} Suppose the blowup is along a face $F= F_1 \cap \ldots \cap F_k$.
 The new sectors that appear correspond to $G_H^{\circ}$
where $H < F_0$. Take any vertex $v$  in $H$. Suppose $v$ projects
to the vertex $w$ of $F$ under the blowdown.  Without loss of
generality assume $w = \bigcap_{j=1}^n F_j $. Then $v =
\bigcap_{0\le j\neq i \le n} F_j $ for some
 $1 \le i \le k$. Without loss of generality assume $i=1$.
 Since $v \le H $, $\mathcal{I}(H) \subset \{0, 2, \ldots, n \}$.
Therefore any $g\in G_H^{\circ}$ may be represented by an element
$\eta= c_0 \lambda_0 + \sum_{j=2}^n c_j \lambda_j $ of $N$ where
each $ c_j \in [0,1)\cap \QQ $. We need to show that the age of
$g$, namely $c_0 + \sum_{j=2}^n c_j $, is an integer.

 But using $\lambda_0 = \sum_{j=1}^k b_j
\lambda_j $ we get that $\eta \in C_w$. In fact
\begin{equation}
\eta = c_0 b_1 \lambda_1 + \sum_{j=2}^k (c_0 b_j + c_j)\lambda_j +
\sum_{j=k+1}^n c_j \lambda_j
\end{equation}
We may write $\eta= \sum_{j=1}^n (m_j + a_j)\lambda_j$ where each
$m_j$ is an integer and each $a_j \in [0,1)\cap \QQ$. Then
$\sum_{j=1}^n a_j \lambda_j $ corresponds to an element of $G_w$.
Since ${\bf X}$ is quasi-$SL$, $\sum_{j=1}^n a_j$ must be an
integer. Therefore $\sum_{j=1}^n (m_j + a_j) $ is an integer.
Hence $c_0 b_1 + \sum_{j=2}^k (c_0 b_j + c_j) + \sum_{j=k+1}^n c_j
$ is an integer. Using $\sum_{j=1}^k b_j = 1$, this yields that
$c_0 + \sum_{j=2}^n c_j$ is an integer.
\end{proof}

\begin{theorem}\label{thmeuler} The Euler characteristic of
Chen-Ruan cohomology is preserved under
 a crepant blowup of a quasi-$SL$ quasitoric orbifold.
 \end{theorem}

 \begin{proof} Let $\rho: Y \to X$ be a crepant blowdown along a face
 $F= \bigcap_{j=1}^k F_j$ of
 $P$. Let $w$ be any vertex of $P$ and let $v_1, \ldots, v_k$ be
 the vertices of $\widehat{P}$ such that $\rho(v_i)= w$.
Suppose $w= \bigcap_{1\le j \le n} F_j$. Then $v_i= F_0 \cap
\bigcap_{1 \le j \neq i \le n} F_j$.

The contribution of $w$ to $\chi_{CR} ({\bf X})$ is $o(G_w)= |\det
\Lambda_{w}|$, see \eqref{ordgv}. The contribution of each $v_i$
to $\chi_{CR} ({\bf Y})$ is $o(G_{v_i}) = |\det \Lambda_{v_i} | =
b_i |\det \Lambda_{w}| = b_i o(G_w) $. As the blowdown is crepant,
we have $o (G_w) = \sum_{i=1}^k o (G_{v_i})$. The theorem follows.
\end{proof}

\subsection{Orbifold $K$-groups}\label{kgps} Orbifold $K$-theory is the $K$-theory
of orbifold vector bundles. Adem and Ruan  \cite{[AR]} proved that there is an
isomorphism of groups between orbifold $K$-theory and $\ZZ_2$-graded orbifold
cohomology theory of any reduced differentiable orbifold, with field coefficients.
Almost complex structure is not necessary for this result as the grading for
orbifold cohomology is the ordinary grading.
For a quasi-$SL$ quasitoric orbifold, since the degrees of cohomology classes as
well degree shifting numbers are even integers,  $K^0_{orb}$ has rank same as
the Euler characteristic of Chen-Ruan cohomology
 and $K^1_{orb}$ is trivial.
Hence by Theorem \ref{thmeuler}, the orbifold $K$-groups are
preserved under crepant blowup of quasi-$SL$ quasitoric orbifolds.

\subsection{Betti numbers} We prove a stronger version
of McKay correspondence, namely the invariance of Betti numbers of
Chen-Ruan cohomology under crepant blowdown, when dimension of
${\bf X}$ is less or equal to six. A more restrictive result was
proved for dimension four in \cite{[GP]}.

\begin{theorem}\label{thmmckay6} Suppose $\rho: Y \to X$ is a crepant
blowdown of quasi-$SL$ quasitoric orbifolds of dimension $\le 6$.
Then the Betti numbers of Chen-Ruan cohomology of ${\bf X}$ and
${\bf Y }$ are equal.
\end{theorem}

\begin{proof} Assume that $\dim({\bf X})=6$. Note that there are no
facet sectors as every characteristic vector is primitive.
Therefore the twisted sectors correspond to either vertices or
edges. The age of a vertex sector is either $1$ or $2$ and such a
sector contributes a generator to $H^2_{CR}$ or $H^4_{CR}$
respectively. An edge sector always has age $1$. Since such a
sector is a sphere it contributes a generator to $H^2_{CR}$ as
well as $H^4_{CR}$. There is only one generator in $H^0_{CR}$ and
$H^6_{CR}$ coming from the untwisted sector. Therefore $h^0_{CR}$
and $h^6_{CR}$ are unchanged under blowup. If $h^2_{CR}$ changes
under blowup then by Poincar\'e duality, $h^4_{CR}$ must change
by the same amount. That would contradict the conservation of
Euler characteristic. Therefore all Betti numbers are unchanged.

The proof for dimension four is similar.
\end{proof}

\begin{lemma}\label{thmh2} Suppose $\rho: Y \to X$ is a crepant
blowdown of quasi-$SL$ quasitoric orbifolds of dimension $\ge 8$.
Then $h^2_{CR}({\bf Y}) \ge h^2_{CR}({\bf X}) $.
\end{lemma}

\begin{proof} The sectors that contribute to $h^2_{CR}$ are the untwisted
sector and twisted sectors of age one. Each age one sector  contributes one
to $h^2_{CR}$. The untwisted sector contributes $h^2$. It is proved in \cite{[PS]}
that $h^2 = m-n $ where $m$ is the number of facets and  $n$ is the dimension of the
polytope.

 Suppose the blowup is along a face
$F$.  The twisted sectors that may  get affected by the blowup are the ones
that intersect $X(F)$. These must be of the form $(S,g)$ where $g$
  belongs to $\bigcup_{w} G_w$ where $w$ varies over vertices of $F$.
 Consider any such $w$. Suppose $\lambda_1, \ldots, \lambda_n$ are the
 corresponding characteristic vectors. Note that the age one
 sectors of $X$ coming from $G_w$ belong to the set
\begin{equation}
A_w = \{ \sum_{j=1}^n a_j \lambda_j: \sum_{j=1}^n a_j= 1 \}
\end{equation}
 Since $\lambda_1, \ldots,
\lambda_n $ are linearly independent, there exists a unique vector
$v $ such that the dot product $\langle \lambda_i , v \rangle = 1$
for each $i$. Hence $A_w $ is a hyperplane given by
\begin{equation}\label{hyp}
A_w = \{ x\in N\otimes \RR: \langle x , v \rangle = 1 \}.
\end{equation}

Note that since the blowup is crepant,
$\lambda_0 \in A_w \cap C_F \cap N$.
The sector corresponding to $\lambda_0$ is lost under the blowup.
However the loss in $h^2_{CR}$ because of it is compensated by the
 contribution from the untwisted sector on account of the new facet $F_0$.

Consider any other age one sector $g$ of ${\bf X}$ in $G_w$. $C_w$ is partitioned into
$n$ sub-cones  by the introduction of $\lambda_0$.
 Accordingly $g$  may be represented
by $\sum_{0 \le j\neq i \le n} c_j \lambda_j$ with  each
$c_j\ge 0$, for some $1\le i \le n$.
This means that $g$ becomes a sector of $Y$ coming from $G_{v_i}$ where
$v_i= \bigcap_{0\le j \neq i\le n} F_j$.
Now $g \in A_w$ as it is an age one sector of ${\bf X}$. Also each $\lambda_j\in A_w$.
Therefore by \eqref{hyp}, $ \sum_{0\le j\neq i \le n} c_j = 1$. This implies that
each $0 \le c_j <1$ and age of $g$ as a sector of ${\bf Y}$ is one as well. The
lemma follows.
\end{proof}

\subsection{Example}\label{exapl} We will consider the weighted
projective space ${\bf X} = \PP(1,3,3,3,1)$ which is a toric
variety. The generators of the one dimensional cones of the fan of
$X$ are $ e_1 =(1,0,0,0)$, $e_2=(0,1,0,0)$, $e_3=(0,0,1,0)$, $e_4=
(0,0,0,1)$ and $e_5 =(-1,-3,-3,-3)$. ${\bf X}$ may be realized as
a quasitoric orbifold with the $4$-dimensional simplex as the
polytope and the $e_i$s as characteristic vectors. However
$\PP(1,3,3,3,1)$ is not an $SL$ orbifold and this choice of
characteristic vectors coming from the fan does not make it an
omnioriented quasi-$SL$ quasitoric orbifold. So we choose a
different omniorientation.

To be precise, by the correspondence established in \cite{[LT]},
we can consider ${\bf X}$ as a symplectic toric orbifold with a
simple rational moment polytope $P$ whose facets have inward
normal vectors $e_1, \ldots, e_5$. The moment polytope may be
identified with the orbit space of the torus action. The
denominations of the polytope are related to the choice of the
symplectic form and is not important for us. Denote the facet of
$P$ with normal vector $e_i$ by $F_i$. We assign the
characteristic vectors as follows
\begin{equation}
\lambda_i = \left\{ \begin{array}{ll} e_i & {\rm if}\, 1\le i \le
4 \\ -e_5 & {\rm if}\, i=5. \end{array} \right.
\end{equation}

The singular locus of ${\bf X}$ is the subset $X(F)$ where $F= F_1
\cap F_5$. The group $G_F$ is isomorphic to $\ZZ_3$ and
\begin{equation}G_F^{\circ} = \{ g = \frac{2}{3} \lambda_1 + \frac{1}{3}
\lambda_5 , g^2 = \frac{1}{3} \lambda_1 + \frac{2}{3} \lambda_5 \}
= \{(1,1,1,1), (1,2,2,2) \}. \end{equation} Thus there are only
two twisted sectors $S(F, g)$ and $S(F, g^2)$, each of age one.
Since $F$ is a triangle, the $4$-dimensional quasitoric orbifold
${\bf X}(F)$ has $h^0= h^2= h^4=1$. Therefore each twisted sector
contributes one to $h^k_{CR}({\bf X})$ for $k= 2,4,6$.

We consider a crepant blowup ${\bf Y}$ of ${\bf X}$ along $X(F)$
with $\lambda_0= (1,1,1,1)$.
 The singular locus of ${\bf Y}$
equals $Y(H)$ where $H= F_0 \cap F_5$. $G_H \cong \ZZ_2$ and
$G_H^{\circ} = \{h = \frac{1}{2} \lambda_0 + \frac{1}{2} \lambda_5
\} = \{(1,2,2,2)\}$. The age one twisted sector $S(H,h)$
contributes one to $h^k_{CR}({\bf Y})$ for $k= 2,4,6$. But
$h^2_{CR}({\bf Y})$ also has an additional contribution from the
new facet. Therefore $h^2_{CR}({\bf Y})= h^2_{CR}({\bf X})$. Then
by Poincar\'e duality, $h^6_{CR}$ are also equal. Finally by
conservation of Euler characteristic we get equality of
$h^4_{CR}$.

It is also possible to directly ascertain the change in the
ordinary Betti numbers due to blowup.
 The new facet $F_0$ is diffeomorphic
to $F \times [0,1]$. So the new polytope has three extra vertices.
We can arrange them to have indices $1, 2, 3$ and keep indices of
other vertices unchanged, see \cite{[PS]} for definition of index.
 This means that ordinary homology, and therefore cohomology, of $Y$ is
richer than that of $X$ by a generator in degrees $2,4,6$.

If we perform a further blowup of ${\bf Y}$ along $H$ with
$(1,2,2,2 )$ as the new characteristic vector, we obtain a
quasitoric manifold $Z$. It is easy to observe that Betti numbers
of Chen-Ruan cohomologies of ${\bf Y}$ and $Z$ are equal. If we
switched the choice of characteristic vectors for the two blowups,
McKay correspondence for Betti numbers would still hold.

Finally consider other choices of omniorientation that could make
${\bf X}$ quasi-$SL$. Switching the sign(s) of $\lambda_2$,
$\lambda_3$ or $\lambda_4$ does not affect quasi-$SL$ness or the
calculations of Betti numbers. Another option is to take
$\lambda_1 = -e_1$ and $\lambda_5 = e_5$. The calculations for
this choice are analogous to the ones above.

\section{Ring structure of Chen-Ruan cohomology}\label{ring}

We will follow \cite{[CH]} and define the structure of an associative ring on
Chen-Ruan cohomology of an omnioriented quasitoric orbifold.

 The normal bundle of a characteristic suborbifold has an almost complex
 structure determined by the omniorientation. More generally suppose
 $F= \bigcap_{i=1}^k F_i$ is an arbitrary face of $P$.
  The normal bundle of the
 suborbifold ${\bf S}(F,g)$, see section \ref{PD}, decomposes into the direct sum of complex orbifold line bundles
 $L_i$  which are restrictions of the normal bundles corresponding to facets $F_i$ that contain $F$.
  Each of these line bundles $L_i$ have a Thom form $\theta_i$. (Note that the Thom forms of
  ${\bf X}(F)$ and ${\bf S}(F,g)$ in ${\bf X}$ may differ at most by a constant factor.)
 For any  $g = \sum_{0\le i \le k} a_i\lambda_i \in Box_F^{\circ}$ define the formal form (twist factor)
 \begin{equation}\label{tg}
 t(g) = \prod_{1 \le i \le k}\theta_i^{a_i}.
 \end{equation}
The order of the $\theta_i$s in the above product is not
important. The degree of $t(g)$ is defined to be $2\iota(g)$.
 For any invariant form $\omega$ on ${\bf S}(F,g)$ define a corresponding twisted form $\omega  t(g)$.
 Define the degree of $\omega  t(g)$ to be the sum of the degrees of $\omega $ and $t(g)$.
Define  \begin{equation} \Omega^{p}_{CR}(F,g) = \{ \omega t(g) \mid \omega \in  \Omega^{\ast}({\bf S}(F,g)), \deg(\omega t(g))= p  \}.
         \end{equation}

  Define the de Rham  complex
 of twisted forms by
 \begin{equation}
 {\Omega^{p}}_{CR}= \bigoplus_{F\le P, g \in Box_F^{\circ}} \Omega^{p}_{CR}(F,g)
 \end{equation}
with  differential
\begin{equation}
d( \sum \omega_i t(g_i))= \sum d(\omega_i) t(g_i).
 \end{equation}
  It is easy to see that the  cohomology of this complex coincides with the
   Chen-Ruan cohomology defined in section \ref{crc}.

  Now we define a product $\star: \Omega^{p_1}_{CR}(K_1,g_1) \times \Omega^{p_2}_{CR}(K_2,g_2) \to
  \Omega^{p_1+p_2}_{CR}(K, g_1g_2) $ of twisted forms as follows,
  \begin{equation}\label{wedge}
  \omega_1 t(g_1) \star  \omega_2 t(g_2)= i_1^{\ast} \omega_1 \wedge  i_2^{\ast}  \omega_2 \wedge \Theta(g_1,g_2)   t(g_1g_2).
  \end{equation}
  Here $K$ is the unique face such that $(K_1 \cap K_2) \le K$ and $g_1g_2 \in G_K^{\circ}$.
  The map $i_j$ is the inclusion of  ${\bf X}(K_1 \cap K_2)$ in ${\bf X}(K_j)$. The form $\Theta(g_1,g_2)$ is obtained as follows.

  Consider the product $t(g_1) t(g_2)$.
   We can think of the $g_j$s as elements of $Box_{v}$ where $v$ is a vertex of $K_1 \cap K_2$.
   Write $g_j= \sum_{i=1}^n a_{ij} \lambda_i$. Write the twist factor
   $t(g_j)$ as  $\prod_{1\le i \le n} \theta_i^{a_{ij}}$.
    A term in the product  $t(g_1) t(g_2)$ looks
   $\theta_i^{a_{i1} +a_{i2}}$. We may ignore the $i$'s for which both $a_{i1}$ and
   $a_{i2}$ are zero. Then there can be three cases:
\begin{enumerate}
\item
  $a_{i1} +a_{i2} <1$.  Then $\theta_i^{a_{i1} +a_{i2}}$ contributes to $t(g_1 g_2)$.
\item
 $ a_{i1} +a_{i2} >1$.
 Then fractional  part $\theta_i^{a_{i1} +a_{i2}-1}$ contributes to $t(g_1 g_2)$ and the integral part is the Thom form
 $\theta_i$ which contributes as an invariant $2$-form to $\Theta(g_1,g_2)$.
 \item
$  a_{i1} +a_{i2} = 1$.
When this happens   $g_1 g_2 \in Box_K^{\circ}$ where $(K_1 \cap K_2)< K$ and
$\theta_i$ contributes to $\Theta(g_1,g_2)$.
\end{enumerate}

If case (3) does not occur for any $i$, then $K= K_1 \cap K_2$ and $ i_1^{\ast}\omega_1 \wedge i_2^{\ast} \omega_2 \wedge \Theta(g_1,g_2)$
 restricts to ${\bf S} (K, g_1g_2 )$ without problem. If case (3) occurs for some $i$'s then the product of the restrictions of
  corresponding $\theta_i$s
 to ${\bf X}(K)$ is, up to a constant factor, the
Thom form of the normal bundle of ${\bf X}(K_1 \cap K_2) $ in ${\bf X}(K)$. The wedge of this Thom form with
 $i_1^{\ast}\omega_1 \wedge i_2^{\ast} \omega_2$
and the restriction of the contributions from case (2) to  ${\bf X}(K) $ defines a form on ${\bf X}(K)$. Thus the star product
is well-defined.

 We extend the star product to a product on ${\Omega^{\ast}}_{CR}$ by bilinearity.
  The differential acts on the star product as follows,
  \begin{equation}
  d( \omega_1 t(g_1) \star  \omega_2 t(g_2))= d( \omega_1 t(g_1)) \star  \omega_2 t(g_2)
   +(-1)^{\deg(\omega_1) + \deg(\omega_2)} \omega_1 t(g_1) \star d(\omega_2 t(g_2)).
  \end{equation}
  Hence the star product induces a product on the Chen-Ruan cohomology.

Observe that the form $ i_1^{\ast} \omega_1 \wedge  i_2^{\ast}  \omega_2 \wedge \Theta(g_1,g_2) $
 is supported in a small neighborhood
of $X(K_1 \cap K_2)$. Therefore the star product of three forms
 $\omega_i t(g_i) \in \Omega^{p_i}_{CR}(K_i,g_i) $, $1\le i \le 3$,
is nonzero only if $K_1 \cap K_2 \cap K_3 $ is nonempty. Now it is fairly
 straightforward to check that the star product is associative.

 {\bf Acknowledgement.} We thank Yongbin Ruan for an useful
suggestion. The second author thanks Cheol-Hyun Cho, Shintaro
Kuroki and Bernardo Uribe for helpful and interesting discussions.
 We both thank la Universidad de los Andes for providing financial
support for our research related activities.

\end{document}